\documentclass[reqno]{amsart}
\usepackage{amsfonts, url}
\usepackage{amsmath,amssymb,latexsym}

\def\marg#1#2{\def\marginnotetextwidth{\the\textwidth}\marginnote{\bf #1}{\bf #2}}

\numberwithin{equation}{section}

\newtheorem{theorem}{Theorem}[section]
\newtheorem{theorem*}{Theorem}
\newtheorem{prop}[theorem]{Proposition}
\newtheorem{lemma}[theorem]{Lemma}
\newtheorem{cor}[theorem]{Corollary}

\theoremstyle{definition}

\newtheorem{question}[theorem]{Question}
\newtheorem{example}[theorem]{Example}
\newtheorem{definition}[theorem]{Definition}
\newtheorem*{definition*}{Definition}
\newtheorem{remark}[theorem]{Remark}

\newcommand{\R}{{\mathbb R}}
\newcommand{\C}{{\mathbb C}}
\newcommand{\Z}{{\mathbb Z}}
\newcommand{\N}{{\mathbb N}}

\newcommand{\GL}{{\operatorname{\bf GL_n}}}

\newcommand{\G}{{\operatorname{\bf G}}}

\newcommand{\W}{{\operatorname{\bf W}}}
\newcommand{\w}{{\operatorname{\bf w}}}
\newcommand{\e}{{\mathfrak{e}}}

\newcommand{\rank}{{\operatorname{rank}}}

\newcommand{\ch}{{\operatorname{char}}}

\newcommand{\Imm}{\operatorname{Im}}

\newcommand{\Hom}{{\operatorname{Hom}}}

\title{Maps}

\author{}

\renewcommand{\GL}{\mathrm{GL}}
\newcommand{\SL}{\mathrm{SL}}
\newcommand{\E}{\mathcal{E}}

\newcommand{\SU}{\mathrm{SU}}
\newcommand{\PSL}{\mathrm{PSL}}
\newcommand{\PGL}{\mathrm{PGL}}
\renewcommand{\G}{\mathrm{G}}

\def\<{\langle}
\def\>{\rangle}
\def\tilde{\widetilde}
\def\o{\tilde{\omega}}
\def\phi{\varphi}

\def \W{\mathcal W}
\def \T{\mathcal T}
\def \H{\mathcal H}

\def \G{\mathcal G}

\def \E{\mathcal E}

\def\e{\epsilon}

\def\a{\mathfrak{a}}
\def\h{\mathfrak{h}}
\def\d{\mathfrak{d}}

\def\tr{\operatorname{tr}}

\def\Aut{\operatorname{Aut}}
\def\w{\tilde{w}}

\def\o{\tilde{\omega}}
\def\ep{\epsilon}
\def\t{\dot t}

\title{Word maps on perfect algebraic groups}
\author[Gordeev, Kunyavski\u\i , Plotkin] {Nikolai Gordeev, Boris Kunyavski\u\i , Eugene Plotkin}

\thanks{The research of the first author
was financially supported by the Ministry of Education and Science of the Russian Federation, project 1.661.2016/1.4. The research of the
second and third authors was supported by ISF grant
1623/16 and the Emmy Noether Research Institute for Mathematics. The paper was finished when the second author visited the MPIM (Bonn).}

\dedicatory{To Boris Isaakovich Plotkin on the occasion of his 90th birthday}

\address{Gordeev: Department of Mathematics, Herzen State
Pedagogical University, 48 Moika Embankment, 191186, St.Petersburg,
RUSSIA and Department of Mathematics of Sankt-Petersburg State University,  198504,  Universitetsky prospekt, 28, Peterhof, St. Petersburg, RUSSIA} \email{nickgordeev@mail.ru}

\address{Kunyavski\u\i : Department of
Mathematics, Bar-Ilan University, 5290002 Ramat Gan, ISRAEL}
\email{kunyav@macs.biu.ac.il}

\address{Plotkin: Department of Mathematics, Bar-Ilan University, 5290002 Ramat Gan,
ISRAEL} \email{plotkin@macs.biu.ac.il}

\thanks{}
\address{}
\email{}

\date{}

\def\e{\epsilon}
\def\a{\alpha}
\def\he{ h}

\begin{document}

\begin{abstract}
We extend Borel's theorem on the dominance of word maps from semisimple algebraic groups to some perfect groups.
In another direction, we generalize Borel's theorem to some words with constants. We also consider the surjectivity
problem for particular words and groups, give a brief survey of recent results, present some generalizations
and variations and discuss various approaches, with emphasis on new ideas, constructions and connections.
\end{abstract}

\maketitle
\section*{Introduction}

The main object of the present paper is a word map $\w\colon G^n\to G$ defined for any word $w=w(x_1,\dots ,x_n)$ from the free group
$F_n = \left<x_1, \dots, x_n\right>$ of rank $n$ and any group $G$. This map evaluates an $n$-tuple $(g_1,\dots ,g_n)$
by substituting in $w$ each $g_i$ instead of $x_i$, $g_i^{-1}$ instead of $x_i^{-1}$ and performing all group operations in $G$.

Our main interest is in studying the image of $\w$, in particular, the question when $\Imm \w = G$,
which, stated in simple-minded manner, consists in understanding whether the equation
\begin{equation} \label{eq-gen}
w(g_1,\dots ,g_n)=a
\end{equation}
with arbitrary right-hand side is solvable in $G$. (Note that the question on the structure of the {\it fibres} of word maps,
which, in the same simple-minded manner, consists in studying the number of solutions of equation \eqref{eq-gen} for varying
right-hand side, is not less interesting, leading to various results of equidistribution flavour, see, e.g.,
\cite{LaS}, \cite{BK}.)

In this paper we consider $G = \mathcal G(K)$ where $\mathcal G$ is a perfect algebraic group defined over a field $K$,
which includes, as a particular case, semisimple groups.
If $K$ is algebraically closed, we will identify $G$ with $\mathcal G$.

The starting point of all such considerations is Borel's theorem \cite{Bo2} saying that for any non-identity word $w$ the induced word map
$\w$ on any semisimple algebraic group is {\it dominant}. (For special words $w$ such as powers and  commutators the question on the surjectivity
of $\w$ was investigated much earlier.)

There are several eventual ways to strengthen Borel's theorem. First, one can try to extend the class of algebraic groups from semisimple groups
to more general ones. In Section  \ref{sec:perf} we consider the case where $G$ is a {\it perfect} algebraic group. For certain subclasses of
such groups we prove the dominance of arbitrary word maps.

Second, although in Borel's theorem $w$ is an arbitrary word, one can try to extend it to a more general case
of {\it words with constants}. This line of research was started in our earlier papers \cite{GKP1}, \cite{GKP2}. In Section \ref{sec:const}
we continue these considerations and prove some more dominance results.

Third, one can try to strengthen dominance results by proving the {\it surjectivity} of relevant word maps.
Usually, to get statements of this kind, one has to pay a certain price by restricting attention either
to particular words, or to particular groups, or to particular fields, or to any combination of above.
In Section \ref{sec:words} we consider the first option and give a brief survey of results available for
power words, commutators, and, more generally, Engel words. In Section \ref{sec:gr} we consider the case of
groups of Lie rank one. Here we generalize some of recent results by Bandman and Zarhin \cite{BZ} and discuss
various approaches which may eventually work in this case.

Throughout, if not stated otherwise, $K$ is an algebraically closed field. We postpone the investigation of
word maps on algebraic groups over some special fields (particularly, local fields) to a forthcoming paper.
We only note that here there is another type of problems, intermediate between the dominance and surjectivity:
namely, one can ask about the structure of the image of the word map with respect to the {\it natural topology}
arising from the ground field under consideration. This is important, for instance, for getting results of ``Waring type''
(we use the terminology introduced by Shalev, see, e.g., \cite{Sh1}), when one is interested in representing every element of $G$
as a product of word values). Whereas in the case where $K$ is algebraically closed such results follow from the dominance
in a straightforward way, much harder work is to be done over smaller fields. One can mention \cite{HLS} for the case where $K$ is a local
field as well as numerous papers devoted to finite groups of Lie type, see. e.g., \cite{Sh1}, \cite{LST} and the references therein,
as well as surveys \cite{BGK}, \cite{Sh2}, \cite{Li}. Returning to the real case, note the striking difference arising for {\it compact} groups $G$,
where the image of a word map may be arbitrarily small in the real topology \cite{Th} (of course, remaining Zarsiki dense).
Some more results for $p$-adic fields were obtained in \cite{GS},
\cite{AGKS}, \cite{HLS}. See \cite{Ku} for a brief survey of results available over global fields.







\bigskip

{\it Notation and terminology}

\medskip

$\N$, $\Z$, $\R$, $\C$ denote the set of natural numbers, the ring of integers, the fields of real and complex and numbers, respectively.

$\mathbb A^n_K$ is the affine space of dimension $n$ over the field $K$. We consider only affine varieties, which are closed irreducible subsets of
$\mathbb A^n_K$. For a variety $X$ and a subset $Y\subset X$ by $\overline{Y}$ we denote the Zariski closure of $Y$ in $X$.
In general, by topology we mean Zariski topology.

As mentioned above, the main object of our consideration is a perfect linear
algebraic group $\G$ defined over a field $K$. This means that the commutator group
$[\G,\G]$ coincides with $\G$. If $K$ is algebraically closed, we identify the group of $K$-points $G= \G(K)$ with $\G$.

The unipotent radical of $\G$ is denoted by $R_u(\G)$.

If a group $H$ acts on a set $X$, the symbol $X^H$ stands for the set of $H$-invariant elements of $X$.

$F_n =\left< x_1, \dots, x_n\right>$ is the free group of rank $n$,
$F_n^0 = F_n$, $F_n^i = [F_n^{i-1}, F_n^{i-1}]$.

When numbering our statements, we use capital letters for known results and numbers for new ones.


\section{Generalities on word maps on linear algebraic groups}


\bigskip

\subsection{$\bf \operatorname{\bf Aut}(F_n)$- and $\bf \operatorname{\bf Aut} (G)$-invariance.} \label{sec1.1} Let $G$ be an abstract group.

Let us start with mentioning the following obvious (and well-known) fact:

$\Imm \w$ is an $\Aut (G)$-invariant subset of $G$.

In another direction, let $\Aut (F_n)$ denote the group of automorphisms of $F_n$.

\begin{prop}
\label{pr1.1}
If $w_1, w_2\in F_n$ lie in the same $\Aut (F_n)$-orbit,
then for any group $G$ the maps
$\w_1, \w_2 \colon G^n\rightarrow G$ have the same image.
\end{prop}

\begin{proof}
Indeed, any group homomorphism $\varphi\colon F_n\to G$ is determined by
the $n$-tuple $(g_1=\varphi (x_1),\dots , g_n=\varphi (x_n))$. Since for any  $w\in F_n$ we have $\varphi(w)=\w(g_1,\dots ,g_n)$, the image of $\w$ coincides with the
set $\{\varphi(w)\}_{\varphi\in\Hom(F_n,G)}$, whence the result.
 \end{proof}

\begin{remark}
Not too much is known about the invariance of $\Imm \w$ with respect to other operations on $F_n$ and $G$.
It would be interesting to divide words into equivalence classes with respect to certain invariance properties of  $\Imm \w$ for a given group $G$.
\end{remark}

\subsection{Semisimple  algebraic groups}
Consider the case where  $G$ is a {\it semisimple} algebraic group.

The basic point for the investigation of word maps on semisimple algebraic groups is the following theorem of A.~Borel \cite{Bo2}.

\bigskip

{\noindent {\bf Theorem A.} {\it Let $G$ be a connected semisimple algebraic group, and let $1\ne w \in F_n$. Then the map
$\w\colon  G^n\rightarrow  G$
is dominant.}}

\begin{cor}
\label{cor2}
Let $w_1\in F_{n_1}, \dots, w_k\in F_{n_k}, k>1,$ be words without common letters, let $w = w_1w_2\dots w_k$, and let
$\w \colon G^{\sum_i n_i}\rightarrow G$
be the corresponding word map. Then $\Imm\, \w = G$.
\end{cor}

\begin{proof}
Since $w_1, \dots, w_k$ are words without common letters, $\Imm \,w = \Imm \,w_1\Imm w_2 \cdots \Imm w_k$. Theorem A implies
that for every $i$ there is a non-empty open subset $U_i $ of $G$ which is contained in $\Imm\, w_i$. Hence $U_1U_2 \cdots U_k \subset \Imm \,w$. But
the product of any two non-empty open subsets of a linear algebraic group coincides with the whole group  (see \cite[Ch.~I, Prop.~1.3]{Bo1}).
\end{proof}

Let $w \in F_n$, and let $\w\colon  G^n\rightarrow G$ be the word map of semisimple algebraic groups.
Then we may view the words $w_1 = w(x^{(1)}_1, \dots , x^{(1)}_n), \dots , w_k = w(x^{(k)}_1, \dots ,x^{(k)}_n)$
as words in different variables. Applying Corollary \ref{cor2}, we get

\begin{cor}
If $k>1$, then
$(\Imm \w)^k = G$.
\end{cor}

\section{Word maps on perfect algebraic groups} \label{sec:perf}
It would be interesting to extend Theorem A to a wider class of algebraic groups. A natural step would be
to assume $G = [G, G]$ to be a {\it perfect} group. However, this works only in particular cases which we are going to describe.

Throughout this section, $\G$ is a connected perfect linear algebraic group defined over an algebraically closed field $K$
of characteristic zero and $G = \G(K)$.
(Recall that we  identify $\G$ with $G$, see Introduction.)

\subsection{General observations}
Put $U = R_u(G)$. Then $G/U$ is a semisimple algebraic $K$-group 
\cite[11.21]{Bo1} and by Mostow's Theorem \cite{Mo} (see, e.g., \cite[Th.~VIII.4.3]{Ho},
\cite[Prop.~5.4.1]{Co} for modern exposition), there exists a closed linear algebraic
subgroup $H$ of $G$ (called a Levi subgroup) isomorphic to $G/U$. (Equivalently, $G=HU$ is a semidirect product.) All Levi subgroups
are conjugate. We fix one of them and denote by $H$ throughout below.

Let
$$U_1 = U,    \, U_2 = [U, U_1], \dots , U_i = [U,  U_{i-1}],\dots ,U_{r+1} =\{1\}
$$
be the lower central series of $U$, and let
$V_i = U_i/U_{i+1}$ denote its quotients.
Then we may view $V_i$ as a $K[H]$-module (indeed, the action of $H$ on $V_i$
induced by conjugation of $U$ by elements of $G$ is $K$-linear because $\ch \,K = 0$).

\begin{definition}
We say that a $K[H]$-module $M$ is {\em augmentative} if it has no $K[H]$-quotients $M/M^\prime$ on which
$H$ acts trivially.
\end{definition}

If $G$ is a perfect group, $V_1$ is an augmentative $K[H]$-module \cite{GS}, \cite{Gor3}.

\begin{definition}
We say that $G$ is a {\em firm} perfect group if $V_i$ is an augmentative $K[H]$-module for every $i$.
\end{definition}

\begin{remark}
Note that if the nilpotency class of $U$ is equal to one, that is, if $U$ is an abelian group, then any perfect group $G$ is firm.
\end{remark}

\begin{definition}
We say that $G$ is a {\it strictly firm} perfect group if $V_i^T = \{0\}$ for every $i$ where $T$ is a maximal torus of $G$.
\end{definition}

\begin{example}
Let $G$ be a simple classical algebraic group (that is of type $A_r$, $B_r$, $C_r$, or $D_r$),
and let $\Pi = \{\a_1, \dots, \a_r\}$ be its standard simple root system (with the notation of \cite[Planches]{Bou}).
Further, let $P_k = P_X$ be the standard parabolic subgroup of $G$ which corresponds to the set of simple roots $X =\Pi \setminus \{\a_k\}$.
For the case $A_r$, $r > 1$ , the group $[P_k, P_k]$ is a strictly firm perfect group for every $k$.
For all remaining cases, the same is true for $r > 2$ and $k > 2$. 
Indeed, it is enough to show that there are no positive roots $\beta$ orthogonal to every root $\alpha_i \in X$.
This follows, in its turn, from the following observation: since $k > 2$, such a root $\beta$ is orthogonal
to $\a_1 = \e_1 -\e_2$ and $\a_2 = \e_2-\e_3$, and therefore $\beta$ is either $\e_i\pm \e_j$, or $\e_i$, or $2\e_i$,
where $3 < i \leq k, j >i$. But then $\beta$ is not orthogonal to $\a_{i-1} = \e_{i-1}-\e_i$.
\end{example}

\begin{theorem} \label{th1.1}
Let $G$ be a connected perfect algebraic group defined over an algebraically closed field $K$ of characteristic zero.  Then
\begin{itemize}
\item[(i)] if $G$ is strictly firm, then for any $1\ne w \in F_n$  the map
$\w\colon G^n\rightarrow G$ is dominant;
\item[(ii)] if $G$ is firm, then for any $w = w_1(x_1, \dots, x_n)w_2(y_1, \dots, y_k)\in F_{n+k}$, $w_1, w_2 \ne 1$,
the word map $\w\colon G^{n+k}\rightarrow G$ is dominant.
\end{itemize}
\end{theorem}

\begin{proof}
$\,$

(i) Let $\d\colon U^m \rightarrow U$ be a map such that for every $i$

\begin{itemize}
\item[(I)] $\d(U_i^m)\subset U_i$;
\item[(II)] for any two $m$-tuples $(u_1, u_2, \dots, u_m) \in U_i^m$ and $(u^\prime_1, u^\prime_2, \dots, u^\prime_m) \in U^m_j, j > i $ we have
\end{itemize}
$$
\d(u_1u_1^\prime, u_2u_2^\prime, \dots, u_mu_m^\prime) \equiv \d(u_1, u_2, \dots, u_m) \d(u_1^\prime, u_2^\prime, \dots, u_m^\prime)  \pmod{U_{j+1}}.
$$
Then we may consider the induced maps $\d_i\colon V_i^m\rightarrow V_i$ given by
$$\d_i (v_1, \dots, v_m) \equiv  \d(u_1, u_2, \dots, u_m) \pmod{U_{i+1}}$$
where $(u_1, u_2, \dots, u_m)$ is an $m$-tuple of preimages of $(v_1, \dots, v_m)\in V^m_i = (U_i/U_{i+1})^m$ in $U_i^m$.

\begin{lemma}
\label{lem1.0}
Let $\d\colon U^m \rightarrow U$  be a map satisfying conditions (I), (II).
If $\d_i(V^m_i) = V_i$ for every $i$, then
$\d(U^m) = U$.
\end{lemma}

\begin{proof}
Let $u \in U_i\setminus U_{i+1}$. Since $\d_i(V^m_i) = V_i$, we have $(u_1, u_2, \dots, u_m) \in U_i^m$ such that $\d(u_1, u_2, \dots, u_m) = uu^\prime$ where $u^\prime \in U_{i+1}$. Let $u^\prime \in U_j\setminus U_{j+1}$ for some $j>i$. Then there is $(u^\prime_1, u^\prime_2, \dots, u^\prime_m) \in U^m_j$ such that $\d(u_1^\prime, u_2^\prime, \dots, u_m^\prime) \equiv u^{\prime -1} \pmod{ U_{j+1}}$. Condition (II) implies that
$\d(u_1u_1^\prime, u_2u_2^\prime, \dots, u_mu_m^\prime) = uu^{\prime\prime}$ where $u^{\prime\prime} \in U_{j+1}$.
Acting this way, we can find a preimage of $u$ in $U_i^m$.
\end{proof}

Since $G$ is a strictly firm perfect group, there exists a non-empty open subset $X\subset T$ such that $[t^{-1},V_i] = V_i$
for every $t\in X$ and every $i$ because $V_i$ is a finite-dimensional $K[H]$-module.

Let $\w_H\colon  H^n\rightarrow H$ be the map corresponding to the same word $w$. Then $\w_H$ is dominant according to the Borel Theorem, and therefore
there exists an open subset $Y\subset T$ such that every element $t \in Y$ has a non-empty preimage $\w_H^{-1}(t)$.

Let $t \in X\cap Y$. Then

\begin{itemize}
\item[(a)] there is an $n$-tuple $(h_1, \dots, h_n) \in H^n$ such that $\w_H(h_1, \dots, h_n) = t$;
\item[(b)] $[t^{-1},V_i] = V_i$ for every $i$.
\end{itemize}

We have $t\in \Imm \w$  because $G$ contains $H$.
Then for every $u\in U$ we have $u t u^{-1}\in \Imm\,\w$. Thus, $t\,[ t^{-1}, u] \in \Imm\w$ for every $u$.
Since $[t^{-1}, V_i] = V_i$ for every $i$, condition (b) and Lemma \ref{lem1.0}
(for $m = 1$ and $\d (u)=[t^{-1},u]$) imply that for every $u^\prime \in U$ there exists $u \in U$ such that $[ t^{-1}, u] = u^\prime$.
Hence $t U \in \Imm\,\w$ for every $t$ belonging to a non-empty open subset $X\cap Y$. Then $\Imm\w$ is dense in $G$.

\bigskip

(ii) We need the following
\begin{lemma}
\label{lem1.5}
If statement (ii) holds for $K = \C$, then it holds for every algebraically closed field $K$ of characteristic zero.
\end{lemma}

\begin{proof}
Since $\ch K = 0$, there exists an extension $F/K$ such that $\C \subset F$. Note that $G = \G(K)$ and $\G(\C)$ are dense subgroups of  $\G(F)$ \cite[18.3]{Bo1}. Hence if we prove that
the image of $\w\colon \G(\C)^{n+k}\rightarrow \G(\C)$ is dense, we also get the density of the image of the map  $\w\colon \G(K)^{n+k}\rightarrow \G(K)$.
\end{proof}

Now we assume $G = \G(\C)$.  We denote by $H_c$ a maximal compact  subgroup of $H$.  It is a real compact Lie group. Consider the action of $H_c$ on
$V_i$ for a fixed $i$.
Since $G$ is a firm group, the natural representation $\mathfrak{i}: H_c\rightarrow \GL(V_i)$ is not trivial, and therefore the image $\mathfrak{i}(H_c)$ is a non-trivial compact subgroup of $\SL(V_i)$ (recall that $\G$ is a connected perfect group). To ease the notation, we will identify the group $\mathfrak{i}(H_c)\leq \SL(V_i)$ with $H_c$.
 We may consider a positive-definite hermitian form $\phi$ on $V_i$ and the group $\SU(V_i)$, which is a maximal compact subgroup of $\SL(V_i)$.
We may and will assume $H_c \leq \SU(V_i)$.

Recall that by Mostow's Theorem, $G$ is a semidirect product $HU$ where $H \leq G$ is a (fixed) Levi subgroup of $G$.
The words $w_1 ,\, w_2$ induce dominant maps $\w_1 \colon H^n\to H,\,\,\,\w_2\colon H^k\rightarrow H$.
By the same arguments as above, we can now get a non-empty open set $X\cap Y\subset T$ such that for every $t \in X\cap Y$ the following properties hold:

\begin{itemize}
\item[(a)] there is an $n$-tuple $(h_1, \dots, h_n) \in H^n$ (respectively, $(h_1, \dots, h_k) \in H^k$) such that

$\w_1(h_1, \dots, h_n)  = t$
(respectively,  $\w_2(h_1, \dots, h_k) = t$);
\item[(c)] $[t^{-1},V_i] = [T, V_i]$ for every $i$.
\end{itemize}

Further, there exists a maximal torus $T^\prime\leq H$ such that $\overline{\left<T, T^\prime\right>} = H$.

Indeed, there are only finitely many  connected  proper closed subgroups $\{\Gamma_i\}$ of $H$ which contain a given
maximal torus $T$ \cite[9.4]{Bo1}. Since the set of all semisimple elements is dense in $G$, there is a maximal torus $T^\prime\leq H$ such that
$T^\prime \nsubseteq \Gamma_i$ for every $i$. Let $\Delta = \overline{\left<T, T^\prime\right>}$, then the identity component $\Delta^0$ contains
$ T^\prime$. Then we have $\Delta^0 \ne \Gamma_i$ for every $i$ but $T \leq \Delta^0$. This implies $\Delta = \Delta^0 = H$.

So for every $t^\prime\in T^\prime$ we have $\overline{\left< t^{\prime-1}T t^\prime, T^\prime\right>} = H$. Hence there is an open set
$Z \subset H\times H$ such that if $(t, t^\prime) \in Z$, then
\begin{itemize}
\item[(1)] $t \in T, t^\prime\in T^\prime$, where $T, T^\prime$ is a pair of maximal tori of $H$ such that $\overline{\left< T, T^\prime\right>} = H$;
\item[(2)] there exist  $(h_1, \dots, h_n) \in H^n, \,\,\,(h^\prime_1, \dots, h^\prime_k) \in H^k$ such that
$$\w_1(h_1, \dots, h_n)= t,\,\,\w_2(h^\prime_1, \dots, h^\prime_k)= t^\prime;$$
\item[(3)] $[t^{-1},V_i] = [T, V_i], [t^{\prime-1},V_i] = [T^\prime, V_i]$ for every $i$.
\end{itemize}

Indeed, obviously, a general pair in $H\times H$ satisfies condition (1). Conditions (2) and (3) follow from (a) and (c).

Since $H_c$ is Zariski dense in $H$, we may assume
that maximal compact tori $T_c\leq T, T_c^\prime \leq T^\prime$ are contained in $H_c$ and there is a Zariski dense subset
$Z_c \subset Z , Z_c \subset H_c\times H_c$ such that for every $(t, t^\prime)\in Z_c$
conditions (1), (2), (3) are satisfied with $T$ replaced by $T_c$ and $T^\prime$ replaced by $T^\prime_c$.

Let now $u, u^\prime \in U_i$. Then
\begin{equation}
\w (u h_1u^{-1}, \dots, u h_nu^{-1}; v^\prime h^\prime_1 u^{\prime-1}, \dots, u^\prime  h^\prime_k u^{ \prime -1}) = ut u^{-1}u^\prime t^\prime u^{\prime -1} = t [t^{-1}, u] t^\prime [t^{\prime -1}, u^\prime] = $$
$$t t^\prime ( t^{\prime -1}[t^{-1}, u] t^\prime )  [t^{\prime -1}, u^\prime] = tt^\prime  ( [t^{\prime -1}t^{-1} t^\prime , t^{\prime -1} u t^\prime]  ) ( [t^{\prime -1}, u^\prime])\in \Imm\, \w.
\label{2.1}
\end{equation}
Consider the map
$$\d\colon U\times U\rightarrow U$$
given by
$\d(u, u^\prime)=[t^{\prime -1}t^{-1} t^\prime , t^{\prime -1} u t^\prime]   [t^{\prime -1}, u^\prime]$. Obviously, the map $\d$ satisfies condition (I).
Let us check  (II). We have the identity $[a, bc] = [a, b] b[a, c]b^{-1}$ that becomes $[a, bc] = [a, b] [a, c]$ in the case when $b$ commutes with $[a, c]$. Further, let $u, u^\prime \in U_i,\,\,\,v, v^\prime \in U_j$ where $j >i$, and let $\bar{u}, \bar{u}^\prime, \bar{v}, \bar{v}^\prime $ be the images of $u, u^\prime \,v, v^\prime $ in $U/U_{j+1}$. Note that $\bar{v}, \bar{v}^\prime$, as well as elements of the form $s\bar{v}s^{-1}, s\bar{v}^\prime s^{-1}, [s,\bar{v}], [s,\bar{v}^\prime]$, belong to the centre of $U/U_{j+1}$. Then using the commutator identity $[a, bc] = [a, b] [a, c]$, we get property (II) for $\d$.
Then $\d$ induces bilinear maps
$\d_i  \colon V_i\times V_i \rightarrow V_i$
given by
$$\d_i(v, v^\prime) = (t^{\prime -1}t^{-1} t^\prime (v) - v) + (t^{\prime -1}( v^\prime) - v^\prime)$$
where $v, v^\prime \in V_i$ are the images of $ t^{\prime -1} u t^\prime, u^\prime \in U_i$ (here we change the multiplicative notation for the operation in $U_i$ to the additive notation for
the operation in $V_i = U_i/U_{i+1}$).

Since $\overline{\left< T_c, T^\prime_c\right>} = H$, we have
$\overline{\left< t^{\prime -1}T_c t^\prime, T^\prime_c\right>} = H$. Therefore, if
$t^{\prime -1}t^{-1} t^\prime (l) - l = 0\,\text{and}\,\,t^{\prime -1}(l) - l = 0$, then
condition (3) implies that $l \in V_i^H$.  Since $\ch K = 0$ and $G$ is a firm perfect group, $V_i^H = \{0\}$. Thus, we get
\begin{equation}
\begin{cases}t^{\prime -1}t^{-1} t^\prime (l) - l =0\\\,\,\, \text{and}\,\,\,\\t^{\prime -1}( l) - l = 0\end{cases} \Rightarrow l = 0.
\label{2.2}
\end{equation}
From (3) we have
\begin{equation}
\Imm\, \d_i = [t^{\prime -1} T t^\prime, V_i] + [T^\prime, V_i].
\label{2.3}
\end{equation}
On the other hand,
\begin{equation}
V_i^{t^{\prime -1}T_c t^\prime} = V_i^{t^{\prime -1}T t^\prime},\,V_i^{T^\prime_c } = V_i^{T^\prime }\,\,[t^{\prime -1}T t^\prime, V_i] =  [t^{\prime -1}T_ct^\prime , V_i], [T^\prime, V_i] = [T^\prime_c, V_i].
\label{2.4}
\end{equation}
If $l, l^\prime$ are vectors of the hermitian space $V_i$ such that $t(l) = \alpha l, \alpha \in \C, \alpha \ne 1$ and $t(l^\prime) = l^\prime$ for some $t\in T_c$,
then $\phi (l, l^\prime)  = 0$ (the same is true for the eigenvectors of the unitary transformation $t^\prime$).
The vector space $V_i^{t^{\prime -1}T_c t^\prime}$ consists of the zero weight vectors of $t^{\prime -1}T_c t^\prime$, the vector space  $ [t^{\prime -1}T_c t^\prime, V_i]$ is spanned by all non-zero weight vectors of $t^{\prime -1}T_c t^\prime$,
and the vector space $V_i^{t^{\prime -1}T_c t^\prime}$ is orthogonal to the vector space   $[t^{\prime -1}T_c t^\prime, V_i]$ . The same holds for the vector spaces $V_i^{T^\prime_c }= [T^\prime_c , V_i]$. Hence
\begin{equation}
V_i^{t^{\prime -1}T_c t^\prime} = [t^{\prime -1}T_c t^\prime, V_i]^\bot , \,\,\,\, V_i^{T^\prime_c } =  [T^\prime_c , V_i]^\bot
\label{2.5}
\end{equation}
where $X^\bot$ is the orthogonal complement to $X$.

Suppose that
\begin{equation}
L_i = [t^{\prime -1} T t^\prime, V_i] + [T^\prime, V_i]\ne V_i.
\label{2.6}
\end{equation}
Since $L_i$ is a subspace of the hermitian space $V_i$ with a positive-definite hermitian form $\phi$, we have
\begin{equation}
V_i = L_i + L_i^\bot.
\label{2.7}
\end{equation}
From \eqref{2.6} and \eqref{2.7} we conclude that $L_i^\bot\ne 0$.
But $$L_i^\bot\stackrel{\eqref{2.6}}{\leq}  [t^{\prime -1}T t^\prime, V_i]^\bot\cap  [T^\prime , V_i]^\bot \stackrel{\eqref{2.4}}{=} [t^{\prime -1}T_c t^\prime, V_i]^\bot\cap  [T_c , V_i]^\bot=$$$$\stackrel{\eqref{2.5}}{=} V_i^{t^{\prime -1}T_c t^\prime}\cap V_i^{T^\prime_c }\stackrel{\eqref{2.4}}{=} V_i^{t^{\prime -1}T t^\prime}\cap V_i^{T^\prime }\stackrel{\eqref{2.2}}{=} \{0\}.$$
This contradicts assumption \eqref{2.6}. Hence $L_i = [t^{\prime -1} T t^\prime, V_i] + [T^\prime, V_i] =  V_i$, and from \eqref{2.3} we get

\begin{equation}
\Imm \d_i = L_i = [t^{\prime -1} T t^\prime, V_i] + [T^\prime, V_i] =  V_i.
\label{2.8}
\end{equation}

Note that for chosen $t, t^\prime \in H_c$ we have equality \eqref{2.8} for every $i$. Hence Lemma \ref{lem1.0} (with $m = 2$) implies that the map
$\d\colon U\times U \rightarrow U$ given by $\d(u, u^\prime) =  [t^{\prime -1}t^{-1} t^\prime , t^{\prime -1} u t^\prime]  [t^{\prime -1}, u^\prime]$ is surjective.
Now condition (2) for the choice of $t, t^\prime$  and \eqref{2.1} imply that there is a dense subset $I\subset H_c$ such that
$I = \{tt^\prime \,\,\,\mid\,\,\,(t, t^\prime) \in Z_c\}$ and for every $s \in I$ we have $sU\subset \Imm \w$ (this follows from \eqref{2.1} and \eqref{2.8}).
Hence $IU \subset \Imm \w$ and $\overline{IU} = G$. Thus, $\w$ is a dominant map.
\end{proof}

\begin{cor}
Let $G$ be a perfect algebraic group over an algebraically closed field $K$ of characteristic zero.
\begin{itemize}
\item[(i)] If $G$ is strictly firm, then for any $1\ne w \in F_n$ we have $(\Imm \w)^2 = \G$.
\item[(ii)] If $G$ is firm, then for any $1\ne w \in F_n$ we have $(\Imm \w)^4 = \G$.
\end{itemize}
\end{cor}

\begin{remark}
It would be interesting to obtain analogues of (at least some of the) results of this section
for Lie polynomials on perfect Lie algebras, in the spirit of \cite{BGKP}.
\end{remark}

\subsection{Special cases}
In this section we consider more carefully some special word maps
$\w\colon G^n \rightarrow G,\,\,\,\w_H\colon H^n\rightarrow H$ on a connected perfect algebraic group $G$ and its quotient $H = G/U$.
Throughout this section we assume $\ch K = 0$.


Below we consider the linear action of $G$ (induced by conjugation) on vector spaces $V_i = U_i/U_{i+1}$.

\bigskip

Let $ \mathfrak h = (h_1, \dots, h_n) \in H^n$ be an $n$-tuple such that
$\w_H(h_1, \dots, h_n) = s\in H$. Let $(u_1,\dots, u_n) \in  U^n$. Then there is a map $\d_\h \colon U^n\rightarrow U$ such that
$$\w(\he_1u_1, \he_2 u_2, \dots, \he_n u_n) = \underbrace{\w(\he_1, \dots, \he_n)}_{= s \in H} \underbrace{\d_\h(u_1, \dots, u_n)}_{\in U}$$
(the functions $\d_\h(z_1, \dots, z_n)$ can be expressed by formulas containing the variables $z_i$ and the operators of conjugation by the elements $h_i^{\pm 1}$).
Thus, for a fixed $s \in H$ the set of all elements in $\Imm \w$ whose projection onto $H$ is equal to $s$ is the set
$$ \{ s \d_\h(u_1, \ldots, u_n)\,\,\,\mid \,\,\, (u_1, \dots, u_n)\in U^n,\,\,\,\h \in \w_H^{-1}(s)\}.$$
Then
\begin{equation}
\w\,\,\,\text{is a dominant map} \Leftrightarrow  \overline{\bigcup_{\h \in \w_H^{-1}(s)}\Imm\,\d_\h} = U\,\,\,\text{for every}\,\,\,s\,\,\,\text{from some open subset of}\,\,H,\label{2.9}
\end{equation}
\begin{equation}
\w\,\,\,\text{is a surjective  map} \Leftrightarrow  \bigcup_{\h \in \w_H^{-1}(s)}\Imm\,\d_\h = U\,\,\,\text{for every}\,\,\,s\in H. \label{2.10}
\end{equation}


\begin{example}
\label{ex1.1}
Let $w = x^m$, $h \in H$. Then $$w(\he u) = (\he u)^m = \he^m (\he^{-m+1}u\he^{m-1}) (\he^{-m+2}u \he^{m-2})\cdots (\he^{-m +(m-1)}u\he^{m- (m-1)}) u$$
and therefore
$$\d_\h(u) = (\he^{-m+1}u\he^{m-1}) (\he^{-m+2}u \he^{m-2})\cdots (\he^{-m +(m-1)}u\he^{m- (m-1)}) u.$$
\end{example}

\begin{example}
\label{ex1.2}
Let $w = [x, y]$, $\h = (h_1, h_2) \in H^2$. Then we have
$$[\he_1u_1, \he_2u_2] = \he_1u_1\he_2 u_2 u_1^{-1} \he_1^{-1} u_2^{-1} \he_2^{-1} = $$
$$= \he_1\he_2\he_1^{-1}\he_2^{-1} (\he_2h_1\he_2^{-1} u_1\he_2\he_1^{-1}\he_2^{-1})(\he_2\he_1 u_2u_1^{-1}\he_1^{-1}\he_2^{-1})(\he_2 u_2^{-1}\he_2^{-1})$$
and therefore
$$\d_{\h}(u_1, u_2) =  (\he_2\he_1\he_2^{-1} u_1\he_2\he_1^{-1}\he_2^{-1})(\he_2\he_1 u_2u_1^{-1}\he_1^{-1}\he_2^{-1})(\he_2 u_2^{-1}\he_2^{-1}).$$
\end{example}

\begin{theorem}
Let $G$ be a firm perfect algebraic group.
If $w \notin [F_n, F_n]$ and $\ch \,K = 0$,  then the map $\w\colon G^n\rightarrow G$ is dominant.
\end{theorem}

\begin{proof}
The restriction of $\d_\h$ to $U_i$ induces the maps
$$\d_\h(i) \colon V_i = U_i/U_{i+1}\rightarrow V_i= U_i/U_{i+1}.$$

\begin{lemma}
\label{lem1.1}
For every $i$ the map  $\d_\h(i) \colon V_i = U_i/U_{i+1}\rightarrow V_i$  is $K$-linear and can be written as a sum of $K$-linear maps
$$\d_\h(i)  = \partial_\h(i1) +  \partial_\h(i2)+\cdots +  \partial_\h(in)$$
where $\partial_\h(ij)(v_1, \dots, v_n) = \d_h(i)(0, \dots, 0, v_j, 0, \dots, 0)$.
\end{lemma}

\begin{proof}
Let $(v_1, \ldots, v_n) \in V_i^n $, and let $u_i$ be a preimage of $v_i $ in $U_i$. We have
$$\d_\h(i) (v_1, \dots, v_n) \stackrel{def}{=}\d_\h(u_1,\ldots, u_n)\pmod{U_{i+1}}.$$
On the other hand, $\d_\h(u_1,\dots, u_n)$ is a product of elements of the form
$$h_{l_1}^{\pm 1} h_{l_2}^{\pm 1}\cdots h_{l_k}^{\pm 1} u_j h_{l_k}^{\mp 1}\cdots h_{l_2}^{\mp 1}h_{l_1}^{\mp1}.$$
Conjugation of elements of $U_i$ by $h   \,\,( h \in H)\,\,$ induces a $K$-linear map $h\colon V_i \rightarrow V_i$. Thus, we have
$$\d_\h(i) (v_1, \dots, v_n) = \sum_{j=1}^ n\underbrace{ ( \text{a sum of  elements of the form}\,\,h_{l_1}^{\pm 1} h_{l_2}^{\pm 1}\cdots h_{l_k}^{\pm 1}(v_j))}_{:= \partial_\h(ij) (v_1, \dots, v_n)}. $$
\end{proof}

Let now $w \notin [F_n, F_n]$. Then there is a variable $x_i$ such that $w(1, \dots, 1, x_i, 1, \dots, 1) = x_i^m$, $m\ne 0$.
Hence it is enough to prove the statement for the case $w = x^m$.  Example \ref{ex1.1} and Lemma \ref{lem1.1} show that for every $h \in H,\,\,\h = (h)$ and every $i$ we  have
$$\d_\h(i)(v) = h^{m-1}(v) + h^{m-2}(v) +\cdots + h(v) + v = (h^{m-1} + h^{m-2} +\cdots +h +1)(v).$$
The linear operator $(h^{m-1} + h^{m-2} +\cdots +h +1)$ on $V_i$ is not invertible only in the cases when the operator $h$  has eigenvalues $\alpha = \sqrt[m]{1}\ne 1$. Hence there is an open subset $X\subset H$ such that for every $h \in X$ and every $i$ the linear operator $(h^{m-1} + h^{m-2} +\cdots +h +1)$ is invertible.
Thus, for $h \in X$ we have $\d_\h(i)(V_i) = V_i$ for every $i$. Lemma \ref{lem1.0} implies $\d_\h(U) = U$.  Hence $s U \in \Imm \w$ for every $s \in X$, and therefore $XU \subset \Imm \w$. Thus, $\w$ is dominant.
\end{proof}

\begin{theorem}
\label{th1.2}
Let $G$ be a firm perfect algebraic group,  
and let $w = [x, y]$. Then the word map $\w \colon G^2\rightarrow G$ is dominant.
\end{theorem}

\begin{proof}
Example \ref{ex1.2} shows that for every $i$, for every $\h = (h_1,h_2)\in H\times H$ and $(v_1, v_2) \in V_i = U_i/U_{i+1}$,
we have
$$\d_{\h}(i) (v_1, v_2) = h_2h_1h_2^{-1}( v_1) + h_2h_1(v_2)  -  h_2h_1(  v_1) - h_2 (v_2) = $$$$= h_2[ (h_1h_2^{-1}h_1^{-1}(h_1( v_1)) - h_1(v_1)) + (h_1(v_2) - v_2)].$$
The same arguments as in the proof of Theorem \ref{th1.1} show that for every $i$ we have $\d_{\h}(i) (V_i^2) = V_i$  for any pair $(h_1, h_2)$ belonging to a dense subset $X\subset H^2$,  and therefore $\d_\h(U^2) = U$.
Then the set  $Y = \{[h_1, h_2]\,\,\mid\,\,\,h_1, h_2 \in X\}$ is dense in $H$ and $sU\in \Imm\w$ for every $s \in Y$. This proves the statement.
\end{proof}

\begin{remark}
It would be interesting to investigate the extendability of the results
of this section to the case where $\ch K = p > 0$. There are (at least) two subtleties,
even when $K$ is algebraically closed:
first, the representation $G=HU$ as a semidirect product may not exist; second, the
action on $V=G/U$ may not be linear. See \cite{McN1}, \cite{McN2} for details.
\end{remark}

To finish with general considerations, let us mention the intriguing question on the
existence of eventual obstructions to the extendability of Borel's dominance theorem
to perfect groups.

\begin{question}
Do there exist a field $K$, a connected perfect $K$-group $\G$ and a non-identity word
$w\in F_n$ such that the word map $\w\colon (\G (K))^n\to \G(K)$ is not dominant?
\end{question}

\section{Surjectivity of special word maps} \label{sec:words}

For certain words we know more than in general case, particularly,
regarding the surjectivity of the corresponding word map. In this section
we give a brief overview of the power words and Engel words.

\subsection{The power map $w = x^m$}
In the case of semisimple algebraic groups we have a complete answer.

\bigskip

{\noindent {\bf Theorem B.} (Steinberg \cite{Stei}, Chatterjee \cite{Ch2}--\cite{Ch3})
{\it Let $K$ be an algebraically closed field of characteristic exponent $p$
(i.e., $p=1$ if $\ch K =0$ and $p=\ch K$ otherwise).
Let $ G$ be a connected semisimple algebraic group.
Then the map $x\mapsto x^m$ is surjective on $ G$ if and only if $m$
is prime to $prz$, where $z$ is the order of the centre of $
G$ and $r$ is the product of ``bad'' primes.}}

\bigskip

In particular, one can guarantee that $m$-th roots can be extracted
in an arbitrary connected semisimple group of adjoint type over
$\mathbb C$ if and only if $m$ is prime to 30.

The case of unipotent algebraic groups was considered in some detail in
\cite{Ch1}--\cite{Ch3}, \cite{DM}, \cite{BM}. However, the following examples
show that it may not be easy to combine the two cases. Namely, for a perfect
(or even strictly perfect) group $G$ and a word $w=x^m$ such that the power map
$\w$ is surjective on both $H = G/U$ and $U = R_u(G)$, the map $\w$ may or may not be surjective
on $G$.

\begin{example}
Let $G = HU$ where $H = \SL_2(\C)$ and $U$ is the $6$-dimensional irreducible $H$-module.
Let $ m = 3$. Then the word map $x\mapsto x^3$ is surjective
on $H$ and on $U$. The weights of $U$ are $\xi^5, \xi^3, \xi, \xi^{-1}, \xi^{-3}, \xi^{-5}$
where $\xi\colon T\rightarrow \C^*$ is the weight of the natural $2$-dimensional representation
(indeed, the $6$-dimensional representation is the representation on $5$-forms in $2$-variables).
Hence $G$ is a strictly perfect group. Let $\sigma\in H$ be an element of order $3$, and let $g$
be any element of $H$ such that $g^3 = \sigma$. We may assume $g, \sigma \in T$. Then in the basis
consisting of the weight vectors the elements $g$ and $\sigma$ are represented by the
following diagonal matrices:
$$g = diag (\ep_9^5, \ep_9^3, \ep_9, \ep_9^{-1}, \ep_9^{-3}, \ep_9^{-5}) $$
where $\ep_9 = \xi(g) = \sqrt[9]{1}$,
$$\sigma = diag (\ep_3^5, 1, \ep_3, \ep_3^{-1}, 1, \ep_3^{-5})$$
where $\ep_3 = \xi(g^3) = \sqrt[3]{1}$.

For every $m$ we have
$$(gu)^m = g^m N_g(u)$$
where
$$N_g(u) = (g^{1-m} u g^{m-1}) \cdots (g^{-1} u g) u.$$
We may view $N_g$ as a linear operator $N_g\colon U\rightarrow U$ which is equal to the sum of linear
operators $1 + g + g^2 +\cdots +g^{m-1}$. Thus, for $m= 3$ we have
$$N_g = 1 +g + g^2.$$
Let $v \in U$ be a weight vector corresponding to the weight $\xi^3$, then $N_g(v) = (1 + \ep_3 +\ep_3^2) v = 0$,
and therefore there is no $u$ such that
$N_g(u) = v$. Hence $\sqrt[3]{\sigma v} \notin G$, and therefore the map $x\mapsto x^3$ is not surjective on $G$.
\end{example}

\begin{example}
Let $G = HU$ where $H = \PSL_2(\C)$ and $U$ is the $3$-dimensional  (adjoint) irreducible $H$-module.
Then the word map induced by $w=x^m$ is surjective
on $H$ and on $U$. Also, {\em the word map $x\mapsto x^m$ is surjective on $G$.}

\begin{proof}
Let $\sigma\in H$, and let $g\in G$ be such that $g^m = \sigma$. We may assume $m > 1$ and $\sigma \ne 1$.
Consider the linear operator $N_g = 1 + g + \cdots + g^{m-1}$ on $U$. We have
$$N_g(u) = 0\,\,\,\text{for some }\,\,\,u \in U \Leftrightarrow \text{ the  eigenvalues of }\,\,\,g\,\,\, \text{are }\,\,\,\sqrt[m]{1}\Leftrightarrow \sigma = g^m = 1 $$
(note that both non-identity eigenvalues of $g$ are of the same order). Since $\sigma \ne 1$, the map  $N_g \colon U\rightarrow U$
is surjective. Thus, for every $v\in U$ there is $u \in U$ such that $N_g(u) = v$. Hence $\sqrt[m]{\sigma v} = g u \in G$, and therefore the map $x\mapsto x^m$ is surjective on $G$.
\end{proof}
\end{example}

\subsection{Commutator word. Ore's problem.}
In 1951 Oystein Ore proved (see \cite{Or}) that every element $g$ of the alternating group $A_n, n \geq 5, $
is a single commutator $g = [\sigma, \tau], \sigma, \tau\in A_n$.
At the end of the paper he wrote: ``It is possible that a similar theorem holds for any simple group of finite order,
but it seems that at present we do not have the necessary methods to investigate the question.''

The conjecture of O.~Ore (known now as Ore's problem) was proved only in 2010 by
Liebeck, O'Brien, Shalev, Tiep \cite{LOST1}.
The solution indeed used advanced techniques which were, of course, unavailable in 1951: the
classification of finite simple groups, the Deligne--Lusztig theory of characters of finite groups of Lie type, as well as advanced
computer algebra. Also, the authors used solution of Ore's problem for groups of Lie type over fields containing more than 8 elements \cite{EG1}.
The interested reader is referred to the Bourbaki talk by Malle \cite{Ma} for a detailed account of this story.

Thus, now we can say that for every finite simple group $G$ the word map $\w\colon G\times G\rightarrow G$ induced by $w =[x, y]$ is surjective.

Even before Ore's paper, in 1949, M.~Got\^o \cite{Got} proved the surjectivity of the word map $\w\colon G\times G\rightarrow G$ for $w = [x, y]$
and for a simple compact Lie group $G$, which is the group of real points $G = \mathcal G(\R)$ of an anisotropic simple algebraic $\R$-group $\mathcal G$
(see \cite[5.2]{VO}).
This result was later on extended to simple complex Lie groups \cite{PW} and, more generally, to
$G = \mathcal G(K)$ where $\mathcal G$ is a simple algebraic group over an algebraically closed field $K$ of arbitrary characteristic \cite{Re}.
(Note that recently Got\^o's method was applied in a more general context \cite{ET}, \cite{Gor4}, see Section \ref{cox} for some details.)

If $K$ is an infinite field and $\mathcal G$ is a simple, simply connected, split $K$-group, then the word map for $w =[x, y]$ is surjective on $G\setminus Z(G)$ (see \cite{EG1}).
The same is true for a finite field $K$ and any simply connected group  $\G$ \cite{EG1}, \cite{LOST1}
(except for the cases where $G/Z(G)$ is not a finite simple group).

Combining the results mentioned above with a theorem of Blau \cite{Bl} on representing central elements of finite groups of Lie type as commutators, we may formulate

\bigskip

{\noindent {\bf Theorem C}.
 {\it  Let $\mathcal G$ be a simple, simply connected algebraic $K$-group , and let $G = \mathcal G(K)$. Further, let
 $\w \colon G\times G\rightarrow G$ be the word map induced by $w = [x, y]$. 
\begin{itemize}
\item[(i)] If $K$ is an algebraically closed field or a finite field such that $G/Z(G)$ is a simple group, then $\Imm \w = G$.
\item[(ii)] If $K = \R$ and $\mathcal G$ is an anisotropic group, then $\Imm\w = G$.
\item[(iii)] If $\mathcal G$ is  split over $K$, then $\Imm\w \supset G\setminus Z(G)$.
\end{itemize}
}}
\bigskip

{\noindent {\bf Remark D}. In (iii) we take $G\setminus Z(G)$ instead of the whole $G$ because elements of the centre $Z(G)$  may not be representable by
a single commutator of elements of $G$. Say, let $G = \SL_2(\R)$. It is well known that the matrix $-1$ is not a commutator. Indeed, if $-1 = [x, y]$, then it is easy to see that
$x^2 = y^2 = -1$, and therefore $\left< x, y\right> = Q_8$ is a quaternion group. But $Q_8\nsubseteq \SL_2(\R)$.}


\subsection{Commutator width.} For any group $G$ the commutator width $l_c(G)$ is defined
as the smallest $n \in \N\cup\infty$ such that for the word
$$w = [x_1, y_1][x_2, y_2]\cdots [x_n,y_n]$$
the word map $\w\colon G^n\times G^n\rightarrow [G, G]$ is surjective.

\begin{example}
Remark D shows that in Theorem C(iii) we may have $l_c(G) > 1$ where
$G =\mathcal G(K)$ and $\G$ is a simple, simply connected, split
$K$-group. However, every element $\gamma \in Z(G)$ can be
represented as a product $g_1g_2$ where $g_1, g_2 \notin Z(G)$.
According to Theorem C(iii), $g_1$ and $g_2$ are commutators, and
therefore $l_c(G) \leq 2$.
\end{example}

\begin{example}
Let $A$ be a Noetherian commutative local ring with residue field $K$, let $\G$ be an $A$-group scheme, and let $G=\G(A)$.
Suppose that the special fibre $\G_K$ is a simple, simply connected, split $K$-group. We have $\G_K(K)=G/M$ where $M$ is the congruence subgroup of $G$.
If $K$ is big enough (in particular, if $K$ is an infinite field), then every
$g \in G, g \notin Z(G)M$ is a single commutator \cite[Theorem~3]{GS}. The same arguments as in the previous example show that $l_c(G) \leq 2$.
\end{example}

\begin{prop}
Let $G$ be a firm perfect group over a field $K$ of characteristic zero. Then $l_c(G) \leq 2$.
\end{prop}

\begin{proof}
It follows from Theorem \ref{th1.2}.
\end{proof}

\begin{remark}
In the case where $G$ is a finite perfect group, examples with
$l_c(G)>1$ are known long ago (see, e.g., \cite{Is}). Recently, sharp
estimates for $l_c(G)$ have been obtained for quasi-simple groups
\cite{LOST2}. However, for arbitrary finite perfect groups one
cannot hope for general estimates: the commutator width is
unbounded, see \cite{HP}, \cite{Ni}. See \cite{Se}, \cite{Li} for a survey of results
on width with respect to more general words.
\end{remark}

\subsection{Coxeter elements and Engel words} \label{cox}
Coxeter elements were used by Got\^o \cite{Got} for proving the
surjectivity of commutator maps. Recently, this approach was applied
in \cite{ET} and later on in \cite{Gor4} for studying Engel words.
See \cite{EG2} for other applications.

\bigskip

Let $R$ be an irreducible root system. Fix a set of simple roots
$\Pi = \{\alpha_1, \dots, \alpha_n\}\subset R$. Let $W$ be the Weyl group of $R$.
Any product of reflections $w_c=\prod_i w_{\alpha_i}$ where each $\alpha_i \in \Pi$ appears
exactly once is called a {\it  Coxeter element} of $W$ (it is allowed to take reflections $w_{\alpha_i}$ in such a product in any order).

Let $\mathcal G$ be a simple algebraic $K$-group corresponding to the root system $R$ (here $K$ is not necessarily algebraically closed),
and let $\mathcal T$ be a maximal torus of $\mathcal G$. Let $\mathcal N_\mathcal G (\mathcal T)$ be the normalizer of $\T$ in $\G$, then $\mathcal N_\G(\T)/\T \approx W$. For any $w \in W$ let $\dot w$ be a fixed preimage in $\mathcal N_\G(\T)$.  A Coxeter element $w_c$ acts on the Euclidian vector space generated by the simple roots from $\Pi$ without fixed nonzero vectors (see \cite[Planches~I--X]{Bou}).
Thus, any preimage $\dot w_c\in \mathcal N_{\G}(\T)$ commutes only with the centre of $\G$. Hence the homomorphism
$[\dot w_c, *] \colon \T\rightarrow \T$, $t\mapsto [w_c,t]$, has a finite kernel and is therefore surjective
(as a map on the algebraic group $\T$ but not necessarily on $\T(K)$). Hence the map
$$
\Phi_c^m=\underbrace{[\dot w_c, [\dot w_c, \cdots [\dot w_c, *]\cdots ]]}_{m-\text{times}} \colon \T\rightarrow \T,
$$
$$
t\mapsto [\dot w_c, [\dot w_c, \cdots [\dot w_c, t]]]
$$
is surjective for any $m\ge 1$.

Now let $\T$ be a $K$-torus. Denote $T = \T(K)$, then $N_G(T) = \mathcal N_\G(\T)(K)$.

\begin{prop}
\label{pr2.1}
Suppose that the following conditions hold:
\begin{itemize}
\item[(a)] $\dot w_c \in N_G(T)$ for some Coxeter element $w_c \in W$;
\item[(b)] the group $T$ is divisible.
\end{itemize}
Then the map
$$\phi_c^m = [\dot w_c, [\dot w_c, \cdots [\dot w_c, *]\cdots ]] \colon T\rightarrow T$$
is surjective.
\end{prop}

\begin{proof}
The map $\phi_c^m$ is a homomorphism of the infinite divisible abelian group $T$ with finite kernel.
\end{proof}

\begin{cor}
\label{cor2.1}
Suppose that one of the following conditions holds:
\begin{itemize}
\item[(i)] the torus $\T$ is completely split and $K^*$ is a divisible group;
\item[(ii)] $K = \R$ is the field of real numbers and $\G$ is anisotropic.
\end{itemize}
Then the map
$$\phi_c^m = [\dot w_c, [\dot w_c, \cdots [\dot w_c, *]\cdots ]] : T\rightarrow T$$
is surjective.
\end{cor}

\begin{proof}
Condition (i) contains condition (b) of Proposition \ref{pr2.1}, and condition (b) obviously holds.

If $\G$ is an anisotropic $\R$-group, then $G = \G(\R)$ is a connected compact Lie group. Hence $N_G(T)/T \approx W$ \cite[6.9.6]{GG}.
Hence we have $\dot w_c \in N_G(T)$. Note that $T = S^1\times S^1\times \cdots \times S^1$ where $S^1$ is the one-dimensional anisotropic $\R$-torus. Hence $T$ is a divisible group.
\end{proof}

Now let $$w_E^m =  [y, [y, \cdots [y, x]\cdots ]]$$ be an Engel word.

\begin{cor}
\label{cor2.2}
Suppose that the torus $\T$ is completely split and $K^*$ is a divisible group. Then for any Engel word $w_E^m$
the image $\Imm \w_E^m$ of the word $\w_E^m\colon G\times G\rightarrow G$ contains all semisimple elements and all unipotent elements.
\end{cor}

\begin{proof}
The statement for semisimple elements immediately follows from Corollary \ref{cor2.1}(i).

Since $\G$ is a split group, all unipotent elements of $G$ are conjugate to an element of a fixed maximal connected unipotent subgroup $U\leq G$,  which is normalized by the corresponding group $T$.
Since $K^*$ is a divisible group, the field $K$ is infinite and therefore there is a regular semisimple element $t \in T$.
Then $\w_E^m(t, U) = U$ (this follows from the well-known equality $[t, U]= U$ which in its turn can be proved by induction on the nilpotency class of $U$). So all  unipotent elements of $G$ lie in $\Imm \w_E^m$.
\end{proof}

The following fact has been proved in \cite{ET} for the unitary groups and in \cite{Gor4} for the general case.

\begin{cor}
\label{cor3.2}
Suppose that $K = \R$ is the field of real numbers and $\G$ is anisotropic.
Then $\Imm \w_E^m = G$.
\end{cor}

\begin{proof}
This is a direct consequence of Corollary \ref{cor2.1}(ii).
\end{proof}

\begin{remark}
The question on the surjectivity of Engel words is open.
The surjectivity is known for $\SL_2$ (see \cite{BZ}), $\PGL_3$ (only for $w_E^1, w_E^2 $) and the groups of types $B_2, G_2$ (see \cite{Gor4}).
See \cite{BGG} for the surjectivity of Engel words on some finite groups of Lie type.
\end{remark}

\subsection{Fixed point free elements in the Weyl group.}
In the previous section we used only one property of Coxeter elements:
they act without fixed points on the Euclidian space $\E$ generated by a simple root system $\Pi$.
Obviously, the same property is shared by all elements from the conjugacy class $C_{w_c} = \{w w_cw^{-1}\mid  w\in W\}$ of a Coxeter element $w_c$ in the Weyl group (note that all Coxeter elements are conjugate in the Weyl group (see \cite[Ch.~5, Prop.~6.1]{Bou}) but not every element in this conjugacy class is a Coxeter element). Thus, in Proposition \ref{pr2.1} and Corollary \ref{cor2.1} we may replace
the preimages of Coxeter elements $\dot w_c$ with the
preimage $\dot w$ of any  $w \in C_{w_c}$.

Below we point out three examples for fixed point free elements in the Weyl group, which can be used in questions like the surjectivity of word maps.

\begin{example}
Any simple group $\G$ contains a semisimple algebraic subgroup $\H = \H_1H_2\cdots \H_m$ such that every simple component $\H_i$ of $\H$ is of type $A_{r_i}$ and the product of maximal tori $\T_1\T_2\cdots \T_m$  of  the components is a maximal torus $\T$ of $\G$ (see  \cite{Bo2}). Then in Proposition \ref{pr2.1} and Corollary \ref{cor2.1} we may replace the element $\dot w_c$ with an element of the form $\dot w_{1c}\dot w_{2c}\cdots \dot w_{mc}$ where
$w_{ic}$ is a Coxeter element of $\H_i$.
\end{example}

\begin{example}
If $\G$ is not of type $A_r$, $D_{2r+1}$, or $E_6$, in Proposition \ref{pr2.1} and Corollary \ref{cor2.1} we may also
take an element of the
form $\dot w_0$ where $w_0$ is the element of maximal weight in $W$ (indeed, in these cases one can check that $\dot w_0(t) = t^{-1}$ for every $t \in T$).
\end{example}

The following example was used in \cite{HLS}.

\begin{example}
\label{ex759}
Let $R$ be the root system of type $D_r$, and let
$$\alpha_1 = \e_1 -\e_2,\alpha_2 = \e_2 - \e_3, \dots, \alpha_{r-1} = \e_{r-1} - \e_r, \alpha_r = \e_{r-1} +\e_r$$
be the standard simple root system in the notation of \cite{Bou}. Put $\beta = \e_1 - \e_r$. Then the element
$w^*  = w_\beta w_{\alpha_r}w_{\alpha_{r-2}}\cdots w_{\alpha_2} w_{\alpha_1}$ acts freely on $\E$. Indeed, we may take
$\{\e_1, \dots, \e_r\}$ as a basis in $\E$. Then $w^*$ acts as a cyclic permutation of the set
$\{\e_1, \dots, \e_r, -\e_1, -\e_2, \dots, -\e_r\}$, and therefore $w^*$ cannot have fixed nonzero vectors in
$\E = \left< \e_1, \dots, \e_r\right>$.
\end{example}

\section{Word maps on $G = \SL_2(K)$ and $G = \PGL_2(K)$} \label{sec:gr}

\subsection{Semisimple elements.} Let $K$ be an algebraically closed field, and let $G = \SL_2(K)$. Further, let $1\ne w\in F_n$,
and let $\w\colon \SL_2(K)^n\rightarrow \SL_2(K)$ be the corresponding word map. The first observation here is the following theorem
\cite{BZ}.

\bigskip

{\noindent {\bf Theorem E.} {\it Every semisimple element of $\SL_2(K)$  except, possibly, $-1$ belongs to $\Imm\,\w$}.
}


\begin{proof} Here we give a proof which is a little bit different from \cite{BZ}.

Use induction by $n$.
If $n=1$, the statement is obvious. Suppose that the statement holds for every  $w\in F_{n-1}$.
If $w(1, x_2,\dots, x_n)$ is not a trivial word, we may use the induction hypothesis.
Thus, we may assume that $w(1, x_2, \dots, x_n) = 1$. Further, since $\w$ is dominant,
we have $\w(g_1, \dots, g_n) \ne 1$ for some
$(g_1, \dots, g_n)\in G^n$.  We may assume $g_1\notin Z(G)$ and
\begin{equation}
\label{equa5.1}
g_1 = \begin{pmatrix}1&a\cr b&1+ab\cr \end{pmatrix} = \begin{pmatrix}1&0\cr b&1\cr \end{pmatrix}\begin{pmatrix}1&a\cr 0&1\cr \end{pmatrix}.
\end{equation}
Indeed, every non-central element of $\SL_2(K)$ is conjugate to an element of the form (\ref{equa5.1}) (see \cite{EG1}), and therefore we may take the $n$-tuple $(\sigma g_1\sigma^{-1}, \dots, \sigma g_n\sigma^{-1})$ instead of $(g_1, \dots, g_n)$.

Fix these $g_2,\dots ,g_n$.
Let $\Psi\colon K^2\rightarrow K$ be given by

\begin{equation}
\label{equa5.2}
\Psi(x, y) = \tr \left( w \left(\begin{pmatrix}1&x\cr y&1+xy\cr \end{pmatrix}, g_2, \dots, g_n\right)\right)
\end{equation}
(here we take the trace of the matrix $w \left(\begin{pmatrix}1&x\cr y&1+xy\cr \end{pmatrix}, g_2, \dots, g_n\right)$).
The function $\Psi(x,y)$ is polynomial and $\Psi (0,0) = 2$ because $w(1, g_2, \dots, g_n) = 1$, and $\Psi (a, b) \ne 2$ because $w(g_1, \dots, g_n) \ne 1$. Thus, $\Psi(x, y)$ is a non-constant polynomial and therefore $\Imm\, \Psi = K$.
Formula (\ref{equa5.2}) implies that $\Imm\,\Psi \subset \Imm\, \tr\circ \w$. Hence
$\Imm\, \tr \circ \w = K.$
Thus, for every $\alpha \in K$ there is $g \in \Imm\,\w$ such that $\tr g = \alpha$. If $\alpha \ne \pm 2$, the condition $\tr g = \alpha$ determines a semisimple element $g \in G$ up to conjugacy. If $\tr g = \pm 2$, then $g = \pm u$ where $u$ is a unipotent element. Note that $w(1, 1,\dots, 1) = 1$, and therefore $1 \in \Imm\,w$. Hence we have every semisimple element in $\Imm\, \w$ except, possibly, $-1$.
\end{proof}

The following corollary is also contained in \cite{BZ}.

\bigskip

{\noindent {\bf Corollary F.}  {\it Let $G = \PGL_2(K)$, and let $C_u$ be the conjugacy class of a non-trivial unipotent element $u \in G$. Further, let $\w \colon G^n\rightarrow G$ be the word map induced by a non-trivial word $w \in F_n$. Then $\Imm\, w \supset G\setminus C_u$.}
}

\begin{remark}
We do not know  if $-1\in\Imm\,\w$ for every word $w$.  
\end{remark}


Since every simple algebraic group of Lie rank $r$ (which is not of type $A_r, r > 1$,
$D_{2k+1}$, $k > 1$, or $E_6$) contains a product  of $r$ copies of groups of rank one, we also have
the following fact, which is a corollary of Theorem E (see \cite{GKP1}, \cite{GKP2}).

\bigskip

{\noindent {\bf Corollary G.} {\it Let $G$ be a simple algebraic group. Suppose that $G$ is not of type $A_r, r > 1$,
$D_{2k+1}$, $k > 1$, or $E_6$, and let $\w\colon G^m \rightarrow G$
be a non-trivial word map. Then every regular semisimple element of
$G$ is contained in $\Imm \w$. Moreover, for every semisimple $g\in
G$ there exists $g_0\in G$ of order $\leq 2$ such that $gg_0\in
\Imm\,\w$.}
}

\subsection{Problem of unipotent elements.}
For $G = \SL_2(K)$, where $K$ is an algebraically closed field, the question whether or not for every word $w\in F_n$
the image of the word map $\w\colon G^n\rightarrow G$ contains a non-trivial unipotent element is wide open.
This is unknown even in the case $n=2, K = \C$.

This problem is related to the structure of the representation variety $$R (\Gamma_w, \SL_2(K)) =\{ \rho\colon \Gamma \rightarrow \SL_2(\C)\},$$ where $\Gamma_w = F_n/\left< w\right>$ is a group with $n$ generators and one relation $w$. Namely, let
$$\T_w = \{(g_1, \dots, g_n) \in G^n\,\,\,\mid\,\,\, \tr \w(g_1, \dots, g_n) = 2\}, $$$$\W_w =  \{(g_1, \dots, g_n) \in G^n\,\,\,\mid\,\,\,  \w(g_1, \dots, g_n) = 1\}.$$
Then $\T_w$ is the affine variety of the $n$-tuples $(g_1, \dots, g_n) \in G^n$ mapped onto the set of the unipotent elements of $G$, and $\W_w$ is the set of the $n$-tuples
$(g_1, \dots, g_n) \in G^n$ satisfying the relation $w(g_1, \dots, g_n) =1$. Then $\W_w = R (\Gamma_w, \SL_2(K))$ (see \cite[page~4]{LM}).  Note  that
$\W_w \subseteq \T_w$ and
$$\W_w \ne \T_w \Leftrightarrow \text{there exists a non-trivial unipotent element } \,\,\,u \in \Imm\,w.$$
Further, let $\T_w = \cup_j \T_w^j$, $\W_w = \cup_i\W_w^i$ be the decompositions into irreducible components. Then each $\W_i$ is contained  in some $\T_w^j$. Note that $\T_w$ is an equidimensional variety of dimension $3n - 1$ (indeed, $\T_w$ is a hypersurface in $G^n$ and the ring of regular functions on $G^n$ is factorial). Suppose that there exists a component $\W_w^i$ of $\W_w$  such that $\dim \,\W_w^i  < 3n-1$. 
Then we have $\W_w^i\varsubsetneqq \T_w^j$ for some $j$, and therefore $\T_w^j\setminus \W_w \ne \emptyset$ (recall that we consider varieties and therefore we have no embedded components $\W_w^i \subset \W_w^k$).
Thus we have the implication
$$
\text{the existence of an irreducible component}\,\,\,\W_w^i\,\,\,\text{where}\,\,\, \dim \W_w^i < 3n-1\Rightarrow$$
$$\text{the existence a non-trivial unipotent element } \,\,\,u \in \Imm\,\w. $$

We do not know whether the converse implication is true.

\begin{question}
Is it true that if $\Imm\,\w$ contains a non-trivial unipotent element $u$, then
$\dim \, \W_w^i < 3n-1$ for some component $\W_w^i$?
\end{question}

In \cite{GKP1}, \cite{GKP2} the structure of the varieties $\W_w, \T_w$ was described in the case $n = 2, K = \C$ for $w = [x^k, y^m]$, $w=[x, y]^p$ where $p$ is prime and
$w = [ [x, y], x[x, y]x^{-1}]$. In all these cases we have a component of dimension $4$ or $3$, which is strictly less than $3\cdot 2 - 1 = 5$, and therefore there is a non-trivial $u \in \Imm\, w$.
The investigation of the structure $R(\Gamma_w, \SL_2(\C))$ is interesting from other points of view (see, e.g., \cite[Section~7]{LM}).
Perhaps better understanding of this structure might also give a clue to the problem of the existence of non-trivial unipotent elements in $\Imm\,\w$.


\subsection{Magnus embedding.}
Let $F_n = \left< x_1, \dots, x_n\right>$,  and let  $A = \Z[t_1^{\pm1}, \dots, t_n^{\pm 1}, s_1, \dots, s_n]$ where $t_j, s_i$ are algebraically independent variables (over $\Z$).
The map
\begin{equation}
\label{equa5.100}
x_i\rightarrow \zeta_i = \left\{\begin{pmatrix} t_i&s_i\cr 0&t_i^{-1}\cr \end{pmatrix}\right\}
\end{equation}
induces the injective homomorphism
\begin{equation}
\label{equa5.101}
F_n/ F^2_n\hookrightarrow  \Upsilon =  \left\{\begin{pmatrix} \alpha&\beta\cr 0&\alpha^{-1}\cr \end{pmatrix}\,\,\,\mid\,\,\,\alpha \in A^*,  \beta \in A\right\}
\end{equation}
called the Magnus embedding (see \cite{BZ}). In \cite{BZ} the following consequence of the Magnus embedding was established:

\bigskip

{\noindent {\bf Theorem H. } {\it Let  $L$  be a field of characteristic zero, and let $w \in F_n\setminus F_n^2$.
Further, let $G = \SL_2(L)$, and let $\w\colon G^n\rightarrow G$ be the corresponding word map.
Then the set $\Imm\, \w$ contains all unipotent elements.}
}
\bigskip

Using the Morozov--Jacobson Theorem, one can extend Theorem H to groups $G = \G(L)$ where $\G$ is any semisimple group defined over $L$,  see \cite{BZ}.

\bigskip

Also, in \cite{BZ} there is the following corollary of Theorem H.

\bigskip

{\noindent {\bf Corollary I.} { \it  Let $G = \PGL_2(K)$ where $K$ is an algebraically closed field of characteristic zero, and let $w\notin F_n^2$. Then the word map
$\w\colon G^n\rightarrow G$ is surjective.}
}
\bigskip

One cannot extend this result to the case of characteristic $p > 0$. For instance, if $\ch K = p >0$  and $w = x^p$, then there are no non-trivial unipotent elements
in $\Imm\, \w$. However, we may give the following generalization of Theorem H and Corollary I.

\begin{theorem}
Let $L $ be an infinite field, and let $w \in F_n\setminus F_n^2$.
Further, let $G = \SL_2(L)$, and let $\w\colon G^n\rightarrow G$ be the corresponding word map.
Then there exists a finite set of primes $S_w$ such that if $p = \ch \,L \notin S_w$ (in particular, if $p =0$),
then the set $\Imm\, \w$ contains a non-trivial unipotent element.
\end{theorem}

\begin{proof}
 We may assume $w\in F^1_n\setminus F_n^2$ (for the case
$w \notin F^1$ the result is easily reduced to the case $w = x^m$ where the inclusion $u \in \Imm\, \w$ is obvious for the cases where $p\nmid m$).


Also, denote by the same symbol $\w$ the corresponding word map $\SL_2(A)^n\rightarrow \SL_2(A)$ (recall that $A = \Z[t_1^{\pm1}, \dots, t_n^{\pm 1}, s_1, \dots, s_n]$).
Consider the restriction of $\w$ to $\Upsilon$ (see \eqref{equa5.101}):
$Res_{\Upsilon} \w \colon \Upsilon^n\rightarrow\Upsilon$. Since $w\in F^1_n\setminus F_n^2$, the Magnus embedding  (see \eqref{equa5.101}) implies
$$\w (\zeta_1, \ldots, \zeta_n) =\begin{pmatrix} 1& f_w\cr 0 &1\cr \end{pmatrix}\ne \begin{pmatrix} 1& 0\cr 0 &1\cr \end{pmatrix}$$
where $0\ne f_w = f_w(t^{\pm 1}_1, \dots, t^{\pm 1}_n, s_1, \ldots, s_n)\in A$. 
Put
$$S_w = \text{all common  prime divisors of all coefficients of }\, f_w(t^{\pm 1}_1, \dots, t^{\pm 1}_n, s_1, \dots, s_n).$$
Let $p = \ch \,L \notin S_w$. For every $a \in A$ 
denote by $\bar{a}$ the image of $a$ with respect to the natural homomorphism $A\rightarrow A/(p) = (\Z/p\Z)[t_1^{\pm1}, \dots, t_n^{\pm 1}, s_1, \dots, s_n]$.
Since $p \notin S_w$, we have $\bar{f}_w\ne 0$. The matrices $\zeta_i$ may be viewed as matrices over $A/(p)$
because their entries are variables $t_i^{\pm 1}, s_i$ (see \eqref{equa5.100}). Thus, we have the matrix equality over the ring $A/(p)$:
\begin{equation}
\label{equa5.3}
\w (\zeta_1, \dots, \zeta_n) =\begin{pmatrix} 1& \bar{f}_w\cr 0 &1\cr \end{pmatrix}\ne \begin{pmatrix} 1& 0\cr 0 &1\cr \end{pmatrix}.
\end{equation}
Since $L$ is an infinite field, the Laurent polynomial $\bar{f}_\omega (t_1, \dots, t_n, s_1, \ldots, s_n) $ is not identically zero on $(L^*)^{n}\times L^n$. Therefore there is $(\alpha_1, \dots, \alpha_n, \beta_1, \dots, \beta_n) \in (L*)^{n}\times L^n$ such that
$\bar{f}_\omega(\alpha_1, \ldots, \alpha_n, \beta_1, \ldots, \beta_n) \ne 0$. Now we get our statement from \eqref{equa5.3} applying the substitutions
$$\zeta_i \rightarrow \begin{pmatrix}\alpha_i&\beta_i\cr 0&\alpha_i^{-1}\cr \end{pmatrix}. $$
\end{proof}

Together with Theorem E, this gives the following corollary.

\begin{cor}
Let $G = \PGL_2(K)$ where $K$ is an algebraically closed field, and let $w\notin F_n^2$.
Then there exists a finite set of primes $S_w$ such that if $p = \ch \,K \notin S_w$ (in particular, if $p =0$),
then the word map $\w\colon G^n\rightarrow G$ is surjective.
\end{cor}

\section{Word maps with constants} \label{sec:const}

Studying word maps, we are naturally led to extending the set-up by considering words with constants, see \cite{Gor2}, \cite{GKP1}, \cite{GKP2}.
(As a vague analogy, one can think of investigating any functions in $n$ variables and then considering substitutions of constants
instead of some variables.)

\bigskip

Let $G$ be a group, let
$
\Sigma = (\sigma_1, \dots, \sigma_r)$ where $\sigma_i\in
G\setminus Z(G)\,\,\,\text{for every}\,\,\,i = 1, \dots, r,
$
and let $w_1, \dots, w_{r+1} \in F_n$. The expression
$
w_\Sigma = w_1\sigma_1w_2\sigma_2\cdots w_r\sigma_rw_{r+1}
$
is called {\it a word with constants} (or a generalized monomial) if
the sequence $w_2, \dots, w_r$ does not contain the identity word.

Equivalently, one can think of a word with constants as of an element
of the free product $G\ast F_n$, see, e.g., \cite{KT}.

We will view a word $w\in F_n$ as a word with constants $w_\Sigma$
with $\Sigma =\emptyset$ and $w=w_1$.
A word with constants also induces a map
$
\w_\Sigma \colon G^n \rightarrow G.
$

If  $G=\mathcal G(K)$ where $\mathcal G$ is a semisimple algebraic
group, then
$\Imm_{\w_\Sigma}$ is not necessarily Zariski dense in $\G$ as in Borel's Theorem.

\begin{example}
\label{ex3.3}
Let $w_\Sigma$ be a word with constants $\Sigma = (\sigma_1,\dots, \sigma_r)$,
and let $\tau \in \G$. Further, let $w^\prime_{\Sigma^\prime} = w_\Sigma \tau w_\Sigma^{-1}$.
Then $\Imm \w^\prime_{\Sigma^\prime}$ is contained in the conjugacy class of $\tau$, and therefore the map $\w_{\Sigma^\prime}^\prime$
cannot be dominant. The same refers to the word map for the word $w^{\prime\prime}_{\Sigma^{\prime\prime}} = [w_\Sigma, \tau]$,
where the image is equal to $(\Imm \w^\prime_{\Sigma^\prime})\,\,\tau^{-1}$.
\end{example}

Consider some problems related to words with constants.

\subsection{Covering number.}  Consider the word  with constants
$
w_\Sigma = x_1 \sigma_1x_1^{-1} x_2 \sigma_2 x_2^{-1}\cdots x_m \sigma_m x_m^{-1}
$ and the word map $\w_\Sigma\colon G^m\rightarrow G$.
Then $\Imm {\w_\Sigma} = C_1C_2\cdots C_m$ where $C_i$ is the
conjugacy class of $\sigma_i$. The minimal number $d  \in \N$ such that $\Imm \w_\Sigma = G$ for every $\Sigma = (\sigma_1, \dots, \sigma_m)$ with $m > d$ is called the {\it extended covering number} of $G$;
it is called the {\it covering number} under the additional condition $\sigma_1 = \sigma_2= \cdots = \sigma_m$.

If $G = \G(K)$ where $\G$ is a simple algebraic $K$-group, then $\w_\Sigma$ is dominant for $m \geq 2\, \rank \,G+1$ (see \cite{Gor1}),
and therefore it is  surjective for $m \geq 4\, \rank\,(G)+2$. We do not discuss here
numerous computations of precise values of covering numbers for different types of algebraic groups as well as particular cases
where the set $\Sigma$ consists of special elements, reflections, root subgroups, etc. It is worth mentioning the general result
by Nikolov \cite{Ni} saying that the extended covering number of $G$ can be arbitrarily large when $G$ runs over all finite groups
of Lie type.

\subsection{Thompson's Conjecture.}
Thompson's Conjecture asserts that
any finite simple group $G$ contains a conjugacy class $C$ such that $C^2 = G$. The conjecture has been proved for $A_n, n \geq 5$, the sporadic groups and the simple groups of Lie type over fields containing more than  $8$ elements (see \cite{EG1} and the references therein). The existence of such a conjugacy class has also been proved for the cases $G = \G(K)/Z(\G(K))$ where $\G$ is a split, simple, simply connected
algebraic group over an infinite field $K$ or a simple anisotropic group
over $K = \R$ (a compact Lie group) (see \cite{EG1}, \cite{Gor1}, \cite{ET}).

Note that the Thompson conjecture is the question on the surjectivity of word maps with constants induced by $w_\Sigma = x\sigma x^{-1} y\sigma y^{-1}$
(here $\Sigma=\{\sigma\}$).

\subsection{Identities with constants.} For  a simple algebraic group $G$ it may happen that a word map with constants $\w_\Sigma \colon G^n \rightarrow G$
is trivial (that is, $\Imm\w_\Sigma = \{1\}$) for a non-trivial word with constants $w_\Sigma$. Such a word $w_\Sigma$  is called an {\it identity with constants}.
Identities with constants were studied, in particular, in \cite{GM}, \cite{To}, \cite{Gor2}, \cite{Step1}, \cite{Step2}.
Interestingly, they play an important role
in description of some ``big'' subgroups of $G$ (see \cite{Step1}, \cite{Step2}).

A simple group $G$ has identities with constants if and only if the corresponding root system $R$ contains roots of different lengths. Moreover, the constants
are special ``small'' elements \cite{Gor2}. For any simple group with roots of different lengths there is, in any case, at least one example of an identity with constants \cite{Gor2}, \cite{Step1}.
However, there is no description of all such identities.

\subsection{Dimension of the image of general word maps with constants.}
Throughout this and the next section we assume that the ground field is algebraically closed.
One of the first questions to ask when describing word maps with constants is the question on the dimension of the image.
Let $w_\Sigma = w_1\sigma_1w_2\sigma_2\cdots
w_r\sigma_rw_{r+1}$ be a word with constants where $w_i = w_i(x_1, \dots, x_n)\in F_n$ and
$\Sigma = (\sigma_1, \dots, \sigma_r)$ is an $r$-tuple of elements of a semisimple algebraic group $G$.
In \cite{GKP2} we proved that there exists an open subset $U\subset G^r$ such that for all $\Sigma \in U$
the word maps with constants $\w_\Sigma \colon G^n \rightarrow G$ have the images of the same dimension $\d$.

We will not strictly follow the definition of a word map and admit that every $\Sigma \in G^r$ induces the word map with constants $\w_\Sigma$ .
(In fact, if $\Sigma$ contains elements from the centre of $G$, we may have $w_\Sigma = 1$.
But in such a case we put $w_\Sigma (g_1, \dots, g_n) = 1$ for every $(g_1, \dots, g_n) \in G^n$ by definition.)
Then for every $\Sigma \in G^r$ we have
$$\dim \Imm \w_\Sigma \leq \d$$
(see \cite{GKP2}). Thus, for given words $w_1, \dots, w_{r+1}$ there is a ``general'' dimension of the image of the word map with constants  $w_\Sigma = w_1\sigma_1w_2\sigma_2\cdots
w_r\sigma_rw_{r+1}$, which is maximal among all possible dimensions.

We may consider the word
$$w^Y = w^Y(x_1, \dots, x_n, y_1, \dots, y_r) = w_1 y_1^{k_1}w_2y_2^{k_2}w_3 \cdots w_r y_r^{k_r}w_{r+1}\in F_{n+r}$$
such that the word with constants $w_\Sigma = w^Y(x_1, \dots, x_n, \sigma_1, \dots\sigma_r)$ is obtained from $w^Y$
by substituting the $r$-tuple $\Sigma$ instead of $(y_1, \dots, y_r)$. In such a case we can generalize the result mentioned
above to any word $\omega (x_1, \dots, x_n, y_1, \ldots, y_r) \in F_{n+r}$
instead of the words of the special form $w^Y$.

\begin{theorem}
\label{th7.7}
Let $G$ be a simple algebraic group.
Let $\omega = \omega (x_1, \dots, x_n, y_1, \dots, y_r) \in F_{n+r}$ be any non-trivial word in $n+r$ variables.
Then there exists an open subset $U\subset G^r$ such that for all $\Sigma \in U$
the word maps with constants $\w_\Sigma \colon G^n \rightarrow G$ corresponding to the word with constants
$w_\Sigma = \omega(x_1, \dots, x_n, \sigma_1, \dots, \sigma_r)$
have the images of the same dimension $\d$. Moreover, for every $\Sigma \in G^r$ we have
$$\dim \Imm \w_\Sigma\leq \d.$$
\end{theorem}

\begin{proof}
The proof is almost the same as in \cite[Theorem~1.1]{GKP2}.

We have dominant maps
$$
\o\colon G^{n+r}\rightarrow G\,\,\,\,\text{and}\,\,\,p_Y\colon
G^{n+r}\rightarrow G^r
$$
where $p_Y$ is the projection onto the components $n+1, \dots, n+r$.
Consider the map
$$
F\colon G^{n+r}\stackrel{(\o, p_Y)}{\rightarrow} G\times G^r.
$$
Let $X =\overline{\Imm F}\subset G\times G^r$, and let
$p^\prime_Y\colon X \rightarrow G^r$ be the projection onto $G^r$.
Then $p_Y^\prime(X) = G^r$. Indeed, for every $r$-tuple
$\Sigma=(\sigma_1, \dots, \sigma_r) \in G^r$ there is a non-empty
set
$$
Z_\Sigma = \{(\w_\Sigma(g_1, \dots, g_n), \sigma_1, \dots,
\sigma_r)\,\,\,\mid\,\,\, (g_1, \dots, g_n )\in G^n\}\subset X.
$$


There exists an open subset $\mathcal V$ of $X$ such that:
\begin{itemize}
\item[(a)] $\mathcal V\subset \Imm F$,
\item[(b)] for every $v \in \mathcal V$ the dimension of every
irreducible component of the preimage $F^{-1} (v) $ is a fixed
number $\mathfrak{f}$,
\item[(c)] for every $u \in \Imm F$ the dimension of every irreducible component of the preimage $F^{-1}
(u)$ is greater than or equal to $\mathfrak{f}$.
\end{itemize}
Let now $\mathcal U\subset G^r$ be an open subset contained in
$p_Y^\prime(\mathcal V)$, and let  $\Sigma = (\sigma_1, \dots,
\sigma_r)\in \mathcal U$. Let $v\in \mathcal V$ be such that
$p^\prime_Y(v) = \Sigma$. Then $v = (\w_\Sigma(g_1, \dots, g_n),
\sigma_1, \dots, \sigma_r)$ for some $(g_1, \ldots, g_n)\in G^n$,
and the dimension of every irreducible component of the preimage
$F^{-1} (v) $ is equal to $\mathfrak{f}$, see (b). Further, the
Zariski closure $\overline{Z}_\Sigma$ is an irreducible closed
subset of $X$. Indeed, $Z_\Sigma$ is the image of an irreducible
variety under the morphism $F_\Sigma\colon G^n\rightarrow G\times
G^r$ given by the formula $F_\Sigma (x_1, \dots, x_n) =
(w_\Sigma(x_1, \dots, x_n), \sigma_1, \dots, \sigma_r)$. Note that
$v \in Z_\Sigma \cap \mathcal V$. Hence there is an open subset
$\mathcal W$ of $\overline{Z}_\Sigma$ such that $v\in \mathcal W
\subset \mathcal V$. Since $\mathcal W \subset \mathcal V$, the
dimension of every irreducible component of $F^{-1}(v^\prime)$ for
every point $v^\prime \in \mathcal W$ is equal to $\mathfrak{f}$,
see (b). Also, for every $v^\prime \in \mathcal W$ the closed subset
$F^{-1}(v^\prime)\subset G^n\times G^r$ is isomorphic (as an affine
variety) to the closed subset $F^{-1}_\Sigma(v^\prime)\subset G^n$.
Hence the dimension of the general fibre of the morphism $F_\Sigma
\colon G^n\rightarrow Z_{\Sigma}\subset G\times G^r$ is equal to $\mathfrak{f}$, and
therefore
$$\dim \overline{\Imm F_\Sigma} = n \dim G - \mathfrak{f}.$$ The
construction of $F_\Sigma$ shows that $\overline{\Imm F_\Sigma} $ is
isomorphic to $\overline{\Imm \w_\Sigma}$ (the projection of
$G\times G^r$ onto the first component gives this isomorphism).
Hence $\dim \overline{\Imm \w_\Sigma} = n \dim G - \mathfrak{f}$
for every $\Sigma \in \mathcal U$.

\bigskip

Let $\Sigma^\prime = (\sigma^\prime_1, \dots, \sigma^\prime_r) \in
G^r$ (possibly,  $\sigma^\prime_i \in Z(G)$ for some $i$). The maps
$\w_{\Sigma^\prime} \colon G^n\rightarrow G, F_{\Sigma^\prime}\colon
G^n \rightarrow G\times G^r$ have the same fibres. Moreover, these
fibres are also isomorphic to the fibres of the map $F\colon G^n\times G^r\rightarrow
G\times G^r$ which correspond to points of the form
$(\w_{\Sigma^\prime}(g_1, \dots, g_n), \sigma_1^\prime, \dots,
\sigma_r^\prime)$. Since the dimension of every fibre of $F$ is at
least $\mathfrak{f}$ (see (c)), the dimension $\mathfrak{f}^\prime$
of the general fibre of $\w_{\Sigma^\prime}$ is at least
$\mathfrak{f}$. Hence
$$\dim \overline{\Imm \w_{\Sigma^\prime}} = n\dim G -\mathfrak{f}^\prime \leq \mathfrak{d} = n \dim G-\mathfrak{f}.$$
\end{proof}

The following Corollary is a strengthening of \cite[ Corollary~1.4]{GKP2}.

\begin{cor}
Let $\omega = \omega (x_1, \dots, x_n, y_1, \dots, y_r) \in F_{n+r}$ be any non-trivial word in $n+r$ variables.  If
$\omega(x_1, \dots, x_n, 1,\dots, 1) \ne 1$, then  there exists an open subset $U\subset G^r$ such that for every $\Sigma \in U$
the word map with constants $\w_\Sigma \colon G^n \rightarrow G$  corresponding to $w_\Sigma = \omega(x_1, \dots, x_n, \sigma_1, \dots, \sigma_r)$ is dominant.
\end{cor}

\begin{proof}
Indeed, for $\Sigma_0 = (1, \dots, 1) \in \G^r$ the map $\w_{\Sigma_0} \colon \G^n\rightarrow \G$ is dominant according to the Borel Theorem.
Hence for a general map $\w_\Sigma \colon \G^n\rightarrow \G$, by Theorem \ref{th7.7} we have $\dim \overline{\Imm\,\w_\Sigma }\geq \dim \overline{\Imm\w_{\Sigma_0}} = \dim G.$
\end{proof}

\subsection{Word maps with constants and the quotient map.}  Let $G$ be a semisimple algebraic group, let $T$ be a maximal torus of $G$, and let $W$ be the Weyl group of $G$.
Then there is the {\it quotient map} $$\pi \colon G\rightarrow T/W$$ (see \cite[3.1]{SS}) taking every $g\in G$ to the class of its semisimple part $g_s$ in $T/W$
(namely, if $g = g_s g_u$  is the Jordan decomposition and $t = x g_sx^{-1} \in T$ is conjugate to $g_s$, then $\pi(g) = \bar{t}$ where $\bar{t}$
is the class of $t$ in $T/W$).

Let now $\w_\Sigma\colon G^n\rightarrow G$ be a word map with constants. Consider the composition $\pi\circ \w_\Sigma\colon G^n\rightarrow T/W$. If this map is dominant,
then so is the word map with constants $\w^\prime_\Sigma\colon G^{n+1}\rightarrow G$ corresponding to $w^\prime_\Sigma = yw_\Sigma y^{-1}$.
Indeed, $y$ here is a new variable, and therefore $\Imm w^\prime_\Sigma$ contains all elements in $G$ which are conjugate to elements from
$\Imm w_\Sigma$. Since $\pi\circ \w_\Sigma$ is dominant, the set $\Imm w^\prime_\Sigma$ contains an open subset of regular semisimple elements of $G$ and is therefore dense in $G$.
Thus, if the map $\pi\circ \w_\Sigma$ is dominant, the map $\w_\Sigma$ is ``dominant up to conjugacy'', that is,
almost all conjugacy classes of $G$ (except for some closed subset of $G$)  intersect $\Imm \w$.

In Theorem \ref{th7.7} we have the condition $\omega(x_1, \dots, x_n, 1, \dots, 1) \ne 1$ for getting a dominant map
$$\o(x_1, \dots, x_n, \sigma_1, \dots, \sigma_r) \colon  G^n  \rightarrow G$$
for a general $\Sigma = (\sigma_1, \dots, \sigma_r)\in G^r$. If we drop the condition $\omega(x_1, \dots, x_n,$
$ 1, \dots, 1) \ne 1$, we cannot expect the
dominance of the corresponding word map with constants. However, we may hope for the following dichotomy for any word $\omega(x_1, \dots, x_n,$
 $y_1, \dots, y_r)\in F_{n+r}$:
$$\Imm \pi\circ \o(x_1, \dots, x_n, \sigma_1, \dots, \sigma_r) = \begin{cases}\text{either just one point
 for every }\\
 \Sigma = (\sigma_1, \dots, \sigma_r) \in G^r,\\
\text{or}\\
\text{a dense subset in }\,\,\,T/W\,\,\,\text{for}\\ \text{ every }\, \Sigma = (\sigma_1, \dots, \sigma_r) \in U\\\text{ from some open set}\,\,\,U\in G^r.
\end{cases}\eqno(\clubsuit)$$
Perhaps such a dihotomy could be the best possible replacement of Borel's theorem
in the context of word maps with constants.

\begin{remark}
\label{rem8.8} Dichotomy ($\clubsuit$) definitely holds if $\rank
\,G = 1$. Indeed, in this case $\dim T/W =1$, and therefore the
image of the irreducible variety $G^n$ with respect to the morphism
$\pi\circ \o(x_1, \dots, x_n, \sigma_1, \dots, \sigma_r)$ is either
just one point, or a dense subset. We cannot drop the restriction
that the second alternative in ($\clubsuit$) holds only for an open
subset. Indeed, consider the word $\omega(x_1, x_2, y_1, y_2) =
x_1y_1x_1^{-1} x_2 y_2x_2^{-1}$. If $\Sigma =(\sigma_1, \sigma_2)$
is the set of ``small'' elements (say, root subgroup elements
$x_\alpha(s)$), then the image of $\o(x_1, x_2, \sigma_1, \sigma_2)$
is the product of the conjugacy classes $C_{\sigma_1}C_{\sigma_2}$,
each of which has ``small'' dimension. Therefore, if $\dim G > \dim
\overline{C_{\sigma_1}C_{\sigma_2}}$, we have $\overline{\Imm
\pi\circ \o(x_1, x_2, \sigma_1, \sigma_2)} <\dim T/W$.
\end{remark}

\begin{remark}
For a word $\omega(x_1, \dots, x_n, y)$ it has been proved in \cite{GKP2} that for $\Sigma = \sigma\in U $  from some open subset $U\subset G$ the  map
$\pi \circ \o(x_1, \dots, x_n, \sigma)\colon G^n\rightarrow T/W$
is dominant under the condition $\o(1, \dots, 1,y) = 1$.
However,
 the condition $\o(1, \dots, 1,y) = 1$ is rather strong
and
 $\clubsuit$ is  unknown even for the case where the set of constants $\Sigma$ is just one element.
\end{remark}
\begin{example}
Let $K$ be an algebraically closed field of characteristic zero and let $G = \G(K)$ where $\G$ is a simple adjoint $K$-group.
Consider $\omega (x, y) = x^a y^b x^c y^d, d \ne -b$ and assume for simplicity that $a$ is prime to 30.
We have $ \omega(1, \sigma) = \sigma^{b+d} \ne 1$ for almost all  $\sigma \in G$. But the word map with constants
$\pi\circ \o(x, \sigma)\colon G\rightarrow T/W$ is dominant for elements $\sigma$ from some open subset of $G$.
Indeed, we may assume $c = -a$ (otherwise we may use Theorem \ref{th7.7}).
Thus we have the word with constants $(x^a \sigma^b x^{-a}) \sigma^d$. There is an open subset $U\subset G$ such that if $\sigma \in U$,
then $\sigma^b, \sigma^d$ are regular semisimple elements of $G$. Let $C_{\sigma^b}$, $C_{\sigma^d}$ be the conjugacy
classes of $\sigma^b$, $\sigma^d$. Then $G\setminus Z(G)\subset C_{\sigma^b} C_{\sigma^d}$, and therefore
$\pi (C_{\sigma^b}\sigma^d)$ is a dense subset in $T/W$. On the other hand, the assumption on $a$ implies, by Theorem D., that for every $h\in G$ there
is $g\in G$ such that $g^a=h$, so that
$C_{\sigma^b}\sigma^d = \{g^a \sigma^b g^{-a}\sigma^d\,\,\,\mid\,\,g \in G\}$.
\end{example}



\begin{remark}
It is tempting to extend the results for words with constants
mentioned in this and in the previous sections to Lie polynomials with constants on
simple Lie algebras, or, even further, to associative
non-commutative polynomials on matrix algebras, in the spirit of
numerous analogies described, e.g., in \cite{KBKP} for genuine words
(without constants). Note that the latter case naturally includes
some innocent looking problems which are wide open; see, e.g.,
\cite{Sl} for the case of matrix equations of the form
$$
A_mX^m+\dots +A_1X +A_0=0.
$$
Even the case $m=2$ of quadratic matrix equations is tricky enough,
see \cite{Ge}. Even stating reasonable conjectures looks as a
challenge.
\end{remark}

\noindent{\it Acknowledgements.} We thank the referees for careful reading and thoughtful remarks, which were very helpful for improving the original version.

\end{document}